\documentclass[a4paper, 11pt]{article}
\usepackage{graphicx}
\usepackage{subfigure}
\usepackage{pgfplots}
\pgfplotsset{compat=1.8}
\usetikzlibrary{shapes,arrows}
\usepackage[margin=0.5in]{geometry}

\usepackage{amsmath}
\usepackage{amsthm}
\usepackage{amsfonts}
\usepackage{amssymb}
\usepackage{array}
\usepackage{algorithm}
\usepackage{algorithmic}
\usepackage{mathrsfs}

\usepackage{cite}
\usepackage{enumerate}
\usepackage{url}
\usepackage{lipsum}
\usepackage{booktabs} % for \toprule

\usepackage{dsfont}
\usepackage{latexsym}
\usepackage{hhline}
\usepackage{color}
\usepackage{textcomp}
\usepackage{hyperref}
\newtheorem{corollary}{Corollary}
\newtheorem{lemma}{Lemma}
\newtheorem{definition}{Definition}

\newtheorem{proposition}{Proposition}

\makeatletter
\newcommand{\argmin}{\mathop{\text{~argmin~}}}

\DeclareMathOperator{\prox}{prox}

\definecolor{gris}{gray}{0.90}
\definecolor{gris25}{gray}{0.90}
\definecolor{americanrose}{rgb}{1.0, 0.01, 0.24}
\definecolor{bostonuniversityred}{rgb}{0.8, 0.0, 0.0}
\definecolor{shamrockgreen}{rgb}{0.0, 0.62, 0.38}
\definecolor{selectiveyellow}{rgb}{1.0, 0.73, 0.0}
\definecolor{royalblue}{rgb}{0.25, 0.41, 0.88}
\definecolor{ashgrey}{rgb}{0.7, 0.75, 0.71}
\definecolor{burgundy}{RGB}{159,29,53}
\definecolor{darkgreen}{RGB}{18,53,26}
\definecolor{lightblue}{RGB}{102,217,255}
\definecolor{fakeorange}{RGB}{255,140,102}
\definecolor{arylideyellow}{rgb}{0.91, 0.84, 0.42}
\definecolor{bananayellow}{rgb}{1.0, 0.88, 0.21}
\definecolor{gris_f}{gray}{0.35}
\definecolor{bordure}{rgb}{0.09,0.17,0.68}
\definecolor{aquamarine}{rgb}{0.5, 1.0, 0.83}
\definecolor{apricot}{rgb}{0.98, 0.81, 0.69}
\definecolor{babyblue}{rgb}{0.54, 0.81, 0.94}
\definecolor{uipoppy}{RGB}{225, 64, 5}
\definecolor{uipaleblue}{RGB}{96,123,139}
\definecolor{uiblack}{RGB}{0, 0, 0}
\definecolor{decoda}{RGB}{0,153, 0}
\definecolor{lightgreen}{rgb}{0.56, 0.93, 0.56}
\definecolor{blue_f}{rgb}{0.2, 0.2, 0.6}
\definecolor{cinnamon}{rgb}{0.82, 0.41, 0.12}
\definecolor{darkpastelgreen}{rgb}{0.2, 0.75, 0.24}
\definecolor{drab}{rgb}{0.59, 0.44, 0.09}

\newcommand{\ta}{t^{\ast}}
\newcommand{\nua}{\nu^{\ast}}
\newcommand{\Ua}{U^{\ast}}
\newcommand{\ua}{u^{\ast}}
\newcommand{\Ia}{I^{\ast}}

\ifpdf
\hypersetup{
  pdftitle={Computing the proximal operator of the \(\ell_1\) induced matrix norm},
  pdfauthor={J.E. Cohen}
}
\fi

\title{Computing the proximal operator of the \(\ell_1\) induced matrix norm}

\author{Jeremy E. Cohen\thanks{CNRS, Universit\'e de Rennes, Inria, IRISA Campus de Beaulieu, 35042 Rennes, France.}}
\date{}

\begin{document}
\maketitle

\abstract{
In this short article, for any matrix \(X\in\mathds{R}^{n\times m}\) the proximity operator of two induced norms \( \|X\|_1 \) and \( \|X\|_{\infty}\) are derived. Although no close form expression is obtained, an algorithmic procedure is described which costs roughly \(\mathcal{O}(nm)\). This algorithm relies on a bisection on a real parameter derived from the Karush-Kuhn-Tucker conditions, following the proof idea of the proximal operator of the \( \max \) function found in~\cite{Parikh2014Proximal}.}

\section{Introduction}

\subsection{Short problem statement}

In this paper, for a given real matrix \(X\in\mathds{R}^{n\times m}\) and any \(\lambda>0\), the following optimization problem is solved:
\begin{equation}\label{eq:proxl11}
  \text{Find } U^*\in\mathds{R}^{n\times m} \text{ that minimizes } %\underset{U\in\mathds{R}^{n\times m}}{\min}
  ~\underset{j\leq m}{\max} \|U_j\|_1 + \frac{1}{2\lambda}\|U-X\|_F^2
\end{equation}
%\begin{equation}\label{eq:pb}
%  \underset{U\in\mathds{R}^{n\times m}}{\argmin} \max_{j\leq m} \|U_j\|_1 + \frac{1}{2\lambda} \|U - X\|^2_F
%\end{equation}
Solving~\eqref{eq:proxl11} means computing the proximal operator of the \(\ell_1\) induced matrix norm which will be denoted as \(\ell_{1,1}\). Notation \(U_j\) refers to the j-th column of matrix \(U\).

\subsection{Reminders}
\subsubsection{The proximal operator}
The proximal operator~\cite{Moreau1965Proximite}, also called proximity operator, of a convex proper lower-semicontinuous real function \( g: \mathds{R}^p \to \mathds{R}\cup \{+\infty\} \) is defined as
\begin{equation}
  \prox_{\lambda g}(x) = \underset{u\in\mathds{R}^m}{\argmin} g(x) + \frac{1}{2\lambda} \|u - x\|_2^2
\end{equation}
for any positive \(\lambda \) and any natural integer \( p \). It is a useful operator in convex optimization as an extension to projection on convex sets, see for instance~\cite{Beck2017first}[Chapter 10] and references therein. It is also an important tool to study properties of Minimum Mean Square Error estimators~\cite{Gribonval2019bayesian}.

Proximal operators often admit closed-form expression, see for instance the list provided at \url{http://proximity-operator.net/}. It also often happens that, although no closed-form is known, the proximity operator can be obtained at low computational cost by some algorithmic procedure. For instance,
\( \prox_{\lambda\max}(x_1,\ldots,x_m) \) is obtained by bisection.

\subsubsection{Matrix induced norms}
For some integers \(n\) and \(m\), given a norm \( \|~\|_q \) on \( \mathds{R}^{m} \) and \( \mathds{R}^{n} \), the induced norm \( \|~\|_{q,q} \) on \(\mathds{R}^{n\times m}\)  is defined as
\begin{equation}\label{eq:inducednorm}
  \underset{\|b\|_q=1}{\max} \|Xb\|_q ~.
\end{equation}

It is well-known that
\begin{equation}
    \|X\|_{1,1} = \max_{1\leq j\leq m} \|X_j\|_1
\end{equation}
where \(X_j\) is the j-th column of matrix \(X\). Moreover, it is also well-known that \( \|X\|_{\infty,\infty} = \|X^T\|_{1,1} \). Let us denote \( \ell_{q,q} \) the induced \(q\)-norm.

Since induced norms are norms, they are continuous convex forms and therefore admit a single-valued proximal operator. Moreover,
\begin{equation}
  \prox_{\lambda \ell_{1,1}}(X) = \underset{U\in\mathds{R}^{n\times m}}{\argmin} \|U\|_1 + \frac{1}{2\lambda}\|U-X\|_F^2 = \underset{U\in\mathds{R}^{n\times m}}{\argmin}\|U^T\|_\infty + \frac{1}{2\lambda}\|U^T - X^T\|_F^2 = {\prox_{\lambda \ell_{\infty,\infty}}(X^T)}^T
\end{equation}
and the proximal operator of the matrix infinity norm is given trivially by the proximal operator of the matrix 1 norm.
Therefore, in what follows, only the proximal operator of the \(\ell_{1,1}\) norm is considered.

\subsection{A remark on the proof technique}

Before getting to the description of the proposed solution to compute the proximal operator of the matrix \( \ell_{1,1} \) norm, let us stress that other proof techniques could very well have been used instead. In particular, using the fact that for any norm \( \|.\|_{\alpha} \) with dual norm \( \| .\|_{\alpha}^* \), it is known~\cite{Beck2017first}[Theorem 6.46] that
\begin{equation}
  \prox_{\lambda \|.\|_{\alpha}} (x) = x - \lambda P_{\mathcal{B}_{\| .\|_{\alpha}^*}[0,1]}(x/\lambda)
\end{equation}
where \( P_{\mathcal{B}_{\| .\|_{\alpha}^*}[0,1]} \) is the projection on the unit ball of the dual norm \( \| .\|^*_{\alpha} \).

In the case of the \( \ell_{1,1} \) norm, although I could not find this result in the literature, it is not hard to prove that
\begin{equation}
  \|X\|_{1,1}^* = \sum_{j=1}^{m} \max_i |X_{ij}|
\end{equation}

Then, projecting on the unit ball of \( \ell_{1,1}^* \) does not seem particularly difficult, but I found a rigorous proof more difficult to obtain than anticipated. Any further development down this line may prove useful for improving the proposed algorithm described further. Indeed, for the proximal operator of the infinity norm, although it is often found in the literature it can be computed by bisection e.g.~\cite{Parikh2014Proximal}, noting that
\begin{equation}
  \prox_{\lambda \ell_{\infty}}(x) = x -  \lambda P_{\Delta}(|x|/\lambda) \ast \text{sign}(x)
\end{equation}
where \( P_{\Delta} \) is the projection inside the unit simplex, \( \ast \) is the element-wise product and \( \text{sign}(x) \) is the sign function,
the proximal operator of the infinity norm may be computed by a simple projection on the simplex. This can be done in a non-iterative fashion and exactly~\cite{wang2013projection}.

\section{KKT conditions}

Given \(X\in\mathds{R}^{n\times m}\) and some positive \( \lambda \), we are interested in solving Problem~\eqref{eq:proxl11}.
It is a strongly convex problem, and therefore admits a unique solution \( U^* \).

This problem belongs to a wider class refered to as finite minimax problems~\cite{di1993smooth,polak2003algorithms}, where a minimizer over the maximum of a finite collection of functions, often considered differentiable, is sought. Note that here the maximum is taken oven non-differentiable functions, and therefore a large portion of that literature does not straightforwardly apply. Moreover, since Problem~\eqref{eq:proxl11} is a very particular finite minimax problem, deriving a specific algorithm should prove beneficial in terms of computing speed vs precision.
%the Karush-Kuhn-Tucker (KKT) optimality conditions are necessary and sufficient.

Nevertheless, taking inspiration from~\cite{di1993smooth}, let us work on a smooth equivalent problem introducing an auxilliary variable \(t\)
\begin{equation}\label{eq:proxl1equiv}
  \text{Find } (U^*, t^*) \text{ in } \underset{\underset{\|U_j\|_1\leq t}{U\in\mathds{R}^{n\times m},\; t\in\mathds{R}_+}}{\argmin} t + \frac{1}{2\lambda}\|U-X\|_F^2 ~.
\end{equation}
which is still a strictly convex problem with a unique solution \((U^*,t^*)\), but is now constrained.

The Lagrangian for this problem is
\begin{equation}
  \mathcal{L}(U,t,\nu) = t + \frac{1}{2\lambda}\|U-X\|_F^2 + \sum_{j=1}^{m} \nu_j \left(\|U_j\|_1 - t\right)
\end{equation}
with \( \nu_j \) nonnegative for all \(j\).

Let us denote \((\Ua,\ta)\)  an optimal solution to the proximal operator problem~\eqref{eq:proxl11}, and \(\nu^\ast\) an optimal dual variable. The KKT conditions state that
\begin{align}
  \forall j\leq m,\; \|\Ua_j\|_1\leq \ta  \tag{KKT1}\label{eq:kkt1} \\
  \forall j\leq m,\; \nua_j\geq 0 \tag{KKT2}\label{eq:kkt2} \\
  \forall j\leq m,\; \nua_j\left(\|\Ua_j\|_1 - \ta\right) = 0 \tag{KKT3}\label{eq:kkt3} \\
  \sum_{j\leq m} \nua_j = 1 \tag{KKT4}\label{eq:kkt4}\\%First order optimality (1)} \\
  \forall j\leq m, \;\frac{1}{\lambda} (\Ua_j-X_j) \in -\nua_j\partial_{\ell_1}(\Ua_j) \tag{KKT5}\label{eq:kkt5}
\end{align}
where \(\partial_{\ell_1}\) is the sub-differential of the \(\ell_1\) norm.
From first order optimality~\eqref{eq:kkt5}, it can be deduced\footnote{This is quite a well-known result, used for instance to compute the proximal operator of the \(\ell_1\) norm} that
\begin{equation}\label{eq:Unu}
  \forall j\leq m,\, \forall i\leq n, \;
  \ua_{ij} = S_{\lambda\nua_j}(x_{ij}) := {\left[|x_{ij}| - \lambda \nua_{j} \right]}^+\text{sign}(x_{ij})~,
\end{equation}
where \(x_{ij}, u_{ij}\) are elements indexed by \((i,j)\) in respectively \(X\) and \(U\). Function \(S\) is often referred to as the soft thresholding operator.

\paragraph{Uniqueness of $\nu$} Even though it is known that \( (U^*,t^*) \) are unique, this may not be the case for \( \nu* \). Nonetheless, from~\eqref{eq:Unu}, it can be deduced that for any \(j\leq m\), as long as there exist \( l\leq n \) such that \(u^*_{lj}\neq 0\), the optimal dual variables \(\nu_j^\ast\) are uniquely defined as \(\nu_j^* = \frac{|u^*_{lj}|-|x_{lj}|}{\lambda}\). Moreover, it is possible to give a necessary and sufficient condition on \(\lambda\) such that it is guarantied all columns
of \(U^*\) are non-zero (see Proposition~\ref{prop:lambdamax}), namely \(\lambda<\lambda_{\max}\) and $X$ has no zero columns. In what follows, unless specified otherwise, it is therefore always supposed that \( \lambda< \lambda_{\max}\)  and that \(X\) has no zero-columns (which can be removed without loss of generality) so that the triplet \((U^*, t^*, \nu^*)\) is unique.

\paragraph{A first link between $t^*$, $\nu^*$ and $X$} Equation~\eqref{eq:Unu} shows that to compute the proximal operator of the \(\ell_{1,1}\) matrix norm, it is enough to compute the optimal values \(\nua_j \), since the thresholding level depends on the dual variable only. It turns out, it is possible to play with the KKT conditions to link \( \nua_j\) with \(\ta \).
%Let us first obtain a similar, naive, but hardly exploitable, relationship. In the next section, we derive a few simple results that lead to an efficient computation of the proximal operator.

From the complementary slackness~\eqref{eq:kkt3}, we already know that if \(\nua_j>0\), then \( \|\Ua_j\|_1 = t\). Moreover, if \( \nua_j=0\), then the j-th column of \(\prox_{\lambda\ell_{1,1}}(X)\) amounts to \(X_j\)
%, or equivalently, the optimal active set \(\Ia\) is the set \(\{j,\; \nua_j>0\}\).
Therefore, denoting with a slightly abusive notation
\begin{equation}\label{eq:omegadef}
  \Omega_{\lambda\nu_j} = \{i\leq n, \; |x_{ij}|> \lambda\nu_j \},
\end{equation}
using~\eqref{eq:Unu} it holds that
\begin{equation}\label{eq:link}
  \nua_j>0\; \implies \; \ta = \sum_{i\in\Omega_{\lambda\nua_j}} \left(|x_{ij}| - \lambda \nua_j \right) = \sum_{i\in\Omega_{\lambda\nua_j}} |x_{ij}| - \#\Omega_{\lambda\nua_j}\lambda\nua_j
\end{equation}

When computing the proximal operator of the \(\max \) function (\textit{i.e.} as done in~\cite{Parikh2014Proximal}), using the first order optimality condition~\eqref{eq:kkt4} leads to a formula linking \(\ta \) to \(X\), which in turns yields a problem easily solved by bisection. Sadly in our case, the cardinals of the index sets \( \Omega_{\lambda\nua_j} \) prevent such a simplification. Summing for \(j\) and minding \(\nua_j\geq 0\), we have
\begin{equation}
  \sum_{j=1}^{m} \frac{1}{\#\Omega_{\lambda\nua_j}\lambda} {\left[ \sum_{i\in\Omega_{\lambda\nua_j}}|x_{ij}|~ -\ta \right]}^+ = 1
\end{equation}
for non-empty sets \( \Omega_{\lambda\nua_j} \)\footnote{This is ensured by \(\lambda<\lambda_{\max}\), see Proposition~\ref{prop:lambdamax}}, which is does not straightforwardly provide a way to compute \(\nu_j^\ast \). The problem is simpler if one supposes that the sets \(\Omega_{\lambda\nu_j}\) are fixed/known. Indeed for such fixed sets, the left-hand side decreases strictly with \(t\) and the solution for \(t\) can be obtained efficiently by bisection. However, there are a combinatorial amount of such possible sets if one seeks to try them all out.

It appears more work is required to provide an efficient algorithm for computing the desired proximal operator. Hopefully, after close inspection, the problem of estimating \(\ta \) and \(\nua_j\) can be simplified due to a few simple results, examined in the section below. We will see that eventually, a simple bisection will be sufficient to compute the desired proximal operator.

\paragraph{Notations}

Let us pause and discuss useful notations for the rest of this manuscript.
Anticipating on results exposed further in the manuscript, let us denote \((\hat{U},\hat{t},\hat{\nu})\) the output of Algorithm~\ref{alg:Algorithm1} that I claim produces an approximation as closed as desired to the optimal solution. Let us also label as ``active set'' for some parameters \((U,t,\nu)\) the set of columns of \(U\) which \(\ell_1\) norm amounts to \(t\)\footnote{Anticipating on further results, by convention the active set does not include columns that are not thresholded, even if their \(\ell_1\) is exactly \(t\) ``by chance''.}.
Denote \(\Ia \) the active set of the optimal solution, and \( \hat{I} \) the active set of the output of the further proposed algorithm. The complementary of an active set \(I\) is denoted \(\overline{I}\).

\section{Algorithm description}

%A second result is useful for the discussion below.
%\begin{lemma}
%  Let \(j\leq m\) a fixed index, and \(\nu_j^{(1)},\nu_j^{(2)}\) two dual variable values. If \( \nu_j^{(2)}<\nu_j^{(1)} \), then \( \Omega_{\lambda\nu_j^{(1)}} \subseteq \Omega_{\lambda\nu_j^{(2)}} \).
%\end{lemma}
%%TODO: rephrase this?

%\begin{proposition}
%  Suppose \(X\) is sorted (decreasing) column-wise. Then for any \( \nu_j \),  there exist \(k_j\) such that \( \Omega_{\lambda\nu_j} = \{1,\ldots,k_j\} \). In other words, the set of admissible \( \Omega_{\lambda\nu_j} \) is \( \{ [1,k], k\in[1,n] \} \).
%\end{proposition}
%This lemma is obtained by considering the first order optimality condition~\eqref{eq:kkt5}, which implies that the clipping first occurs for the smallest values in each column of \(X\), and therefore increasing the clipping threshold means thresholding entries of \(X\) in increasing intensity order.
%
%Propositions 1 and Lemma 1 combined lead to the following result.
%{\color{red}//}
%\begin{proposition}\label{prop:decrease}
%For all indexes \(j\) such that \(\nu_j>0\), \(t\mapsto \nu_j(t)\) is a strictly decreasing, piece-wise linear function, and so is \( t\mapsto \sum_{j=1}^{m} \nu_j(t) \).
%\end{proposition}

\subsection{Designing the algorithm}

Knowing the proximal operator of the max function is usually computed with bisection on \(t\), it is natural to turn towards this option for computing the proximal operator of the \(\ell_1\) matrix induced norm.

Let us therefore suppose that a value of \(t\) is given. The question is, what
can we do to gain knowledge on the position of \(\ta\) relative to \(t\)? It
turns out the answer is a little intricate. Let us give an informal description
of the procedure. First, following Definition~\ref{def:tas}, from \(t\) one may
define \(I(t)\) as the active set if it happened that \(\ta =t\) (which is very
unlikely at this stage). A nice feature of the KKT conditions in our
problem is that removing first-order optimality conditions~\eqref{eq:kkt4} from
the set of optimality conditions, it is
possible to compute uniquely dual parameters \(\nu(t)\) with active set exactly
\(I(t)\), see Section~\ref{sec:computev}. This is where things get interesting: it is quite simple to show that
\(\ta\) is respectively above or below \(t\) if and only if \(\sum_{j\in
I(t)}\nu_j(t)\) is below or above 1. The intuition behind this last fact is that
as \(t\) grows, the dual variable \(\nu_j(t)\) always and simultaneously
decrease, even if the active set if wrongly estimated. Therefore, the relative
position of \(t\) with \(\ta\) is determined by the sign of \(\sum_{j\in
I(t)}\nu_j(t)-1\). This is further detailed in
Proposition~\ref{prop:okbisection}.

From there, information on the location of \(\ta\) is available, which allows to perform a bisection on \(t\). The procedure stops when a \(t\) has been found sufficiently close to \(\ta\). It is then possible to show, see Proposition~\ref{prop:iterprec}, that variables \((\nu, U) \) are as close as desired to their optimal values. Algorithm~\ref{alg:Algorithm1} is a proposed implementation of this bisection.

\begin{algorithm}
  \caption{Proposed algorithm for computing \(\prox_{\lambda\ell_{1,1}}\), in python numpy notation.}\label{alg:Algorithm1}
  \begin{algorithmic}
   \STATE{\textbf{Inputs}: \(n\times m\) real matrix \(X\), positive \(\lambda \), precision \(\delta\).}
   \STATE{1. Sort \(|X|\) column-wise in decreasing order, store the output in \(Y\).}
   \STATE{2. Sort \(\{ \|X_j\|_1, j\leq m\} \) in decreasing order, store the sorted indexes in a list \(J=[j_1,\ldots,j_m]\).}
   \STATE{3. If \( \lambda \geq \sum_j y_{1j} \) then return \(U=0\), \(t=0\) and
   \(\nu_j = \frac{y_{1j}}{\lambda}\).}
   \STATE{4. Initialize \(t_{\min}=0\), \( t_{\max}=\| X_{j_1} \|_1 \) and \(t = \frac{t_{\max}}{2}\). Set \(\nu_j=0\).}
  \WHILE{\(t_{\max} - t_{\min} > \delta \) }
    \STATE{5. Compute \(I(t)\) as in Definition~\ref{def:tas}.}
    \STATE{6. Compute \(\nu_j(t)\) for \(j\in I(t)\) using Algorithm~\ref{alg:routine} (see Section~\ref{sec:computev}).
    Set \(\nu_j(t)=0\) for all \( j\in\overline{I(t)} \).}
    \STATE{7. If \(\sum_{j\leq m} \nu_j(t) > 1 \) set \(t_{\min}=t\) and increase \(t\) as \(t = \frac{t+t_{\max}}{2}\) }
    \STATE{8. If \(\sum_{j\leq m} \nu_j(t) < 1 \) set \(t_{\max}=t\) and decrease \(t\) as \(t = \frac{t+t_{\min}}{2}\) }
  \ENDWHILE{}
  \STATE{\textbf{Output}: Matrix \(\widehat{U}\) such that \(\widehat{U}_j \approx S_{\lambda\nua_j}(X_j) \), slack variable \(t\approx \ta\), dual variables \(\nu \approx \nua \). }
  \end{algorithmic}
\end{algorithm}

\subsection{Complexity}
It is noticeable that Algorithm~\ref{alg:Algorithm1} involves only a single loop, and sorting or norm computations outside that loop. The column-wise sorting at line 1 has theoretical complexity \( \mathcal{O}(n\log_2(n)m) \). Line 2 has a cost \( \mathcal{O}(nm + m\log_2(m)) \). Inside the while loop, the only non-trivial operation is line 6, which has complexity roughly \( \mathcal{O}(m\log_2(n)) \), see Algorithm~\ref{alg:routine}. The number of iterations in the while loop is directly controlled by \(\delta \). Indeed the search interval length is divided by two at each iteration, such that \(\delta  = \frac{t_{\max}}{2^{k+1}}\) with \(k\) the number of iterations.
Therefore \( k = \log_2(t_{\max}/\delta) -1 \), which should not be a large number for reasonable precision. In particular, for a fixed small value \(\delta\), and supposing that \(t_{\max}\) scales linearly with \(n\), the number of iteration \(k\) should grow roughly as \(\log_2(n)\). Within that framework, the total complexity of the algorithm is \( \mathcal{O}(\log^2_2(n)m + nm + m\log_2(m)) \) with a small constant.

\section{Proving the algorithm works}
In what follows, I show a number of properties revolving around Algorithm~\ref{alg:Algorithm1}.
\begin{itemize}
  \item A connection between \(t\), \( I \) and \( \nu \) when the KKT conditions are partially verified is established in Proposition~\ref{prop:bij}. This connection is essential to understanding why the algorithm works.
  \item The key result relating \(\sum_{j\leq m} \nu_j(t) \) and \(t-\ta\) is established in Proposition~\ref{prop:okbisection}. This will guaranty the convergence of Algorithm~\ref{alg:Algorithm1} towards the optimal \(\ta \).
  \item Since the proposed routine produces an approximation of \(\ta\), I show how precision on \(\ta\) propagates nicely on the other variables \(\nua\) and \( \Ua \) in Proposition~\ref{prop:iterprec}.
  \item An auxiliary algorithm is provided to compute \(\nu_j(t)\) in Section~\ref{sec:computev}, which is an essential part of the proposed algorithm.
  \item Some guidance is provided regarding the choice of lambda. In particular, the maximum \(\lambda\), above which the proximal operator is null, is derived in closed-form.
\end{itemize}

\subsection{Connecting the parameters}

\subsubsection{Placing slack and dual variables on a parametric curve.}

From equation~\eqref{eq:link}, it can be seen that the slack variable \(t\) and dual variables \( \nu_j \) are connected. However, as long as $\lambda$ is small enough, there is a unique pair \((t^\ast,\nu^\ast)\) that satisfies all the KKT optimality conditions.
Intuitively, to obtain a more flexible relationship, the KKT conditions should be relaxed.
Let \(\Sigma(I)\) the set of \((t,\nu)\) that satisfy
dual feasibility~\eqref{eq:kkt2},
complementary slackness~\eqref{eq:kkt3}, first order optimality~\eqref{eq:kkt5} and for which any \(j\in I\) is an active constraint index and the solution to the proximal operator problem is non-trivial, \textit{i.e.}
\begin{equation}
  \Sigma(I) = \{ (t,\{\nu_j\}_{j\in I}) \; | \;  t>0,\; \; \forall j\in I, \nu_j >0, \; \|S_{\lambda\nu_j}(X_j)\|_1=t, \;  \}~.
\end{equation}
where the soft-thresolding operator \(S_{\lambda\nu_j}\) is taken entry-wise.
Moreover, the fact that \(I\) contains only active constraints imposes that \(t<t_{\max}(I) := \min_{j\in I} \|X_j\|_1 \) as shown below in Proposition~\ref{prop:tdomain}. In turn, this conditions automatically ensures that \(\nu_j>0\).
\begin{proposition}\label{prop:tdomain}
  Let \((t,\nu)\in\Sigma(I)\). Then \(0< t < t_{\max}(I)\)  where \(t_{\max}(I)=\min_{j\in I} \|X_j\|_1 \). Conversely, any \((t,\nu)\) such that \(0<t<t_{\max}\) and \(\|S_{\lambda\nu_j}(X_j) \|_1 = t\) for all \(j\) in \(I\) belongs to \(\Sigma(I)\).
  In particular, the globally optimal slack variable \(\ta \) lives in \([0,t_{\max}]\) where \(t_{\max} = \max_{j\leq m} \|X_j\|_1 \).
\end{proposition}
\begin{proof}
  One easily checks that \(t\geq t_{\max}(I)\) is equivalent to the existence of some index \(j\) such that
  \begin{equation}
    \|S_{\lambda\nu_j}(X_j)\|_1\geq \|X_j\|_1,
  \end{equation}
  which happens if and only if \(\nu_j\leq0\). The global result is obtained by noticing that \((\ta,\nua)\in\Sigma(\Ia)\), and taking a pessimistic bound on \(t_{\max}(\Ia)\).
\end{proof}

 In other words, one may write instead
\begin{equation}
  \Sigma(I) = \{ (t,\{\nu_j\}_{j\in I}) \; | \; 0< t<t_{\max}(I)\;\text{and}\; \forall j\in I, \|S_{\lambda\nu_j}(X_j)\|_1=t \}
\end{equation}

Interestingly, the following proposition shows that \(\Sigma(I)\) is a parametric curve in one of the variables. More precisely, under KKT conditions~\eqref{eq:kkt2},~\eqref{eq:kkt4} and~\eqref{eq:kkt5}, \(t\) and all  \(\nu_j\) are in bijection.

\begin{proposition}\label{prop:bij}
  Let \(I\) some index set. There exist strictly decreasing piece-wise linear isomorphisms $\psi_j$ such that %\(\forall t\in[0,t_{\max}(I)],\; \forall j\in I\),
  %\begin{equation}
  %  \psi_j: t \mapsto \psi_j(t)=\nu_j \; %\; \text{and} \; \; \nu_j \mapsto t(\nu_j) ~.
  %\end{equation}
  %are such that
  any \((t,\nu)\) in \(\Sigma(I)\) satisfies \(\nu_j = \psi_j(t)\) %\((t,\{\psi_j(t)\}_{j\in I})\in\Sigma(I)\)
  for all \(j\in I\).
  %Moreover, for any \(j\in I\) and \(\nu_j^{(1)},\nu_j^{(2)}\) two dual variable values in \( \Sigma(I) \), if \( \nu_j^{(2)}<\nu_j^{(1)} \), then \( \Omega_{\lambda\nu_j^{(1)}} \subseteq \Omega_{\lambda\nu_j^{(2)}} \).
\end{proposition}
\begin{proof}
Let \( j\in I \). For any \((t,\nu_j) \) in \(\Sigma(I)\), as in~\eqref{eq:link}, complementary slackness and first order optimality yield
\begin{equation}\label{eq:link2}
  \sum_{i\in\Omega_{\lambda\nu_j}} |x_{ij}|-\lambda \#\Omega_{\lambda\nu_j} \nu_j = t  ~.
\end{equation}
Let \( \sigma_j \) the index permutation that sorts \(X_j\) in increasing order. Because of~\eqref{eq:kkt5}, it is can be observed that the sets \( \Omega_{\lambda\nu_j} \) are telescopic. In particular, if for some \( i\leq n \) it holds that
 \begin{equation}\label{eq:interval}
  \frac{|x_{\sigma_j(i)}|}{\lambda}\leq \nu_j \leq \frac{|x_{\sigma_j(i+1)}|}{\lambda}~,
\end{equation}
then \(\Omega_{\lambda\nu_j} = \{\sigma_j(1),\ldots,\sigma_j(i)\} \). Moreover, on each intervals~\eqref{eq:interval}, equation~\eqref{eq:link2} shows that one can define \(\psi_j^{-1}(\nu_j)=t\) as a decreasing function of \(\nu_j\) with non-zero, finite slope \(\frac{1}{\#\Omega_{\lambda\nu_j}}\). This proves that \(\psi_j^{-1}\)
defines a piece-wise linear, strictly decreasing bijection. By a simple symmetry argument, \(\psi_j: t\mapsto \nu_j\) is therefore also a piece-wise linear strictly decreasing function.

\end{proof}

%For convenience, we abusively note $\nu_j(t):=\psi_j(t)$ and $t(\nu_j):=\psi_j^{-1}(\nu_j)$.

\paragraph{Remark}
For any index set $I$ and any $t$ such that \(0<t<t_{\max}(I)\), it holds that \((t, \{\psi_j(t)\}_{j\in I})\in \Sigma(I)\). In other words, isomorphisms \(\psi_j\) can be defined independently of the choice of I, and may be computed separately for each $j\leq m$.

%This last proposition states that inside \(\Sigma(I)\), parameter \(t\) is in direct correspondence with the dual variables \(\nu_j\). This fact provides a practical procedure to find the optimal solution \( (\ta, \nua, \Ua ) \) by searching for \(\ta \) by bisection as detailed below.

\subsubsection{Slack variables define active sets.}
I have shown above that for any index \(j\leq m\), given a slack variable \(t<t_{\max}\), a subset of the KKT conditions imply that the dual variable \( \nu_j \) should either be null or is entirely determined by \( t \). Moreover, as long as index \(j\) is in the considered active set \(I\), \(\psi_j(t)\) is independent of \(I\). This is somewhat unsatisfying, since one would like to build a connection \(\nu_j(t)\) without having to worry about some index set \(I\).

To that end, I define the active set induced by a given value of \(t\), which is the one we will consider for computing \(\psi_j(t)\).
\begin{definition}\label{def:tas}
  For a given \(t\in[0,t_{\max}]\), define \(I(t) = \{j\leq m, \; \|X_j\|_1>t \}\).
\end{definition}
It is clear that there is a unique such active set \(I(t)\) per value of \(t\). Intuitively, \(I(t)\) would be the optimal active set if one knew that \(t\) was actually \(\ta\). Now, given a value of \(t\) meant to approximate \( \ta \), one should not be interested in other sets than \(\Sigma(I(t))\). Therefore, we may define \(\nu_j(t)\) for any index \(j\leq m\) as follows:
\begin{equation}\label{eq:defnu}
  \nu_j(t):= \left\{
  \begin{array}{ll}
  0  & \text{if } j\notin I(t) \\
  \psi_j(t)  & \text{if } j\in I(t), \text{ as defined in Proposition~\ref{prop:bij}}
  \end{array}\right.
\end{equation}
Also, vector \( \nu(t)\) is the collection of all \(\nu_j(t)\) for \(j\leq m\).

Before tackling the core result that allows for bisection search, let us note that as \(t\) grows, it may approach a boundary \( \|X_j\|_1\) for some \(j\in I(t)\). If \( t \) keeps growing, then \(I(t)\) shrinks since \(j\) may not be in the active set anymore. This observation leads to the following trivial lemma:
\begin{lemma}\label{lemma:telesc}
  If \(t<t'\) then \(I(t')\subseteq I(t)\).
\end{lemma}

  \subsection{Convergence to the optimal slack variable}

Now that all the ingredients have been prepared, we are ready for the main result. In a nutshell, for a given \(t\), the value of \(\sum_{j\leq m} \nu_j(t)\) alone provides information on whether \(\ta\) is ``right'' or ``left'' of \(t\).

\begin{proposition}\label{prop:okbisection}
  Let \(t\in]0,t_{\max}[\), and compute \(\nu(t)\). It holds that \(\text{sign}(\sum_{j\in I(t)}\nu_j(t) - 1) = \text{sign}(\ta - t)\).
\end{proposition}

\begin{proof}
  According to Lemma~\ref{lemma:telesc}, we may always assume that either \(I(t)\subsetneq \Ia\), \(I(t)=\Ia\) or \(\Ia \subsetneq I(t)\). Let us consider each case separately.

  \underline{Case 1: \(I(t)\subsetneq \Ia\)}:\\
  It holds that \(t > \max_{j\in \overline{I(t)}}  \|X_j\|_1 \geq \ta \), where the last inequality comes from the fact that \(\Ia\) must contain the column of \(X\) not already in \(I(t)\) with largest \(\ell_1\) norm according to~\eqref{eq:kkt1}. Therefore, in that case,
  \begin{equation}
    \sum_{j\in I(t)} \nu_j(t) < \sum_{j\in \Ia} \nu_j(t) < \sum_{j\in \Ia} \nua_j =1~.
  \end{equation}
  The first inequality comes from \(I(t)\subsetneq \Ia\) and the second one from \(t>\ta\). The equality is simply~\eqref{eq:kkt4}.

  \underline{Case 2: \(\Ia \subsetneq I(t)\)}:\\
  This case is similar to Case 1. The same reasoning shows that \(\ta > t\), and
  \begin{equation}
    \sum_{j\in I(t)} \nu_j(t) > \sum_{j\in \Ia} \nu_j(t) > \sum_{j\in \Ia} \nua_j =1~.
  \end{equation}

  \underline{Case 3: \(I(t) = \Ia\)}: \\
  It holds that
  \begin{equation}
    \sum_{j\in I(t)} \nu_j(t) = \sum_{j\in \Ia} \nu_j(t) - \nua_j + \nua_j = \left(\sum_{j\in \Ia} \nu_j(t) - \nua_j\right) + 1
  \end{equation}
  and therefore
  \begin{equation}
    \sum_{j\in I(t)}\nu_j(t) - 1 = \sum_{j\in \Ia} \nu_j(t) - \nu_j(\ta)~.
  \end{equation}
  Since all \(\nu_j(t)\) are decreasing functions of \(t\), the sign of \(t-\ta\) is the opposite sign of \(\nu_j(t) - \nu_j(\ta)\) for all \(j\in \Ia\). This concludes the proof.
\end{proof}

\subsection{Algorithm precision}

It has been established in Proposition~\ref{prop:okbisection} that each bisection iteration gets \(\hat{t}\) closer to \(\ta\). Therefore, it is straightforward to obtain a precision on \(\ta\).
\begin{corollary}
  Algorithm~\ref{alg:Algorithm1} finds \(\hat{t}\in \ta\pm \delta\).
\end{corollary}

Moreover, because of the definition of the active set, which does not allow \(\ta\) to be exactly \( \|X_j\|_1 \) for any \(j\leq m\), it is possible to ensure that the estimated active set \(I(\hat{t})\) is exactly \(\ta\).
\begin{corollary}
  Let \(\hat{t}\) the output of Algorithm~\ref{alg:Algorithm1}. Suppose that
  \begin{equation}\label{eq:ascondition}
    \forall t\in [\hat{t}-\delta, \hat{t}+\delta], \; I(t) = I(\hat{t}),
  \end{equation}
  then \(\hat{I} = \Ia\).
\end{corollary}
\begin{proof}
  It is ensured that \(\ta\in [\hat{t}-\delta, \hat{t}+\delta]\), and therefore condition~\eqref{eq:ascondition} implies \( I = \Ia \) when cast for \(t = \ta\).
\end{proof}
Condition~\eqref{eq:ascondition} is easy to check numerically, thus Algorithm~\ref{alg:Algorithm1} can in principle return, on top of approximate primal and dual solutions, a guaranty of active set optimality.

It remains to discuss the precision of Algorithm~\ref{alg:Algorithm1} on the primal variables \(U\) and dual variable \(\nu_j\). At first glance, since \(t\) is the sum of absolute values of \(U_j\) as long as \(j\in I(t)\), there could be huge discrepancies in \(U_j\) while \(t\) is close to \(\ta\). Hopefully, recalling that \(U\) depend directly on \(\nu_j\) which are monotone with \(t\), such error cancelation cannot happen, which leads to the following result:
\begin{proposition}\label{prop:iterprec}
  Algorithm~\ref{alg:Algorithm1} return dual variables \(\hat{\nu}\) and primal variables \(\hat{U}\) such that \(\forall j\leq m\) and \(\forall i\leq n\), \(|\hat{\nu}_j - \nua_j| \leq \frac{\delta}{\lambda} \) and  \( |\hat{u}_{ij} - \ua_{ij} | < \delta \).
  Furthermore, if \(I(\hat{t}) = \Ia\), then \(\forall j\in \overline{I(t)}\), \( \hat{U}_j = \Ua_j \) and \(\hat{\nu}_j = \nua_j\).
\end{proposition}
\begin{proof}
  Let \(j\leq m\). According to equation~\eqref{eq:link}, \(\nu_j(t)\) is a piecewise linear function. The maximal variation of \(\nu_j(t)\) with respect to \(t\) is obtained when \(\#\Omega_{\lambda\nu_j(t)} = 1\). Indeed, the slope of the linear pieces are \(\frac{1}{\lambda\#\Omega_{\lambda\nu_j(t)}}\), and the case \(\#\Omega_{\lambda\nu_j(t)}=0\) is obtained when \(\nu_j(t)=0\) in which case \(\nu_j(t)\) does not vary with \(t\). Therefore, for any \(j\leq m\),
  \begin{equation}
    |\hat{t}-\ta| < \delta \; \implies \; |\nu_j(\hat{t}) - \nua | \leq \frac{\delta}{\lambda}
  \end{equation}
  A similar argument links the variations of \(\nu_j\) and each \(u_{ij}\) for \(i\leq n\). Indeed, first order optimiality condition~\eqref{eq:kkt5} dictates that \(u_{ij}\) is piecewise-linear (there are in fact only two pieces) with respect to \(\nu_j\), and in particular the slope is at most \(\lambda\). Therefore,
  \begin{equation}
  |\hat{t}-\ta| < \delta \; \implies \; |\hat{u}_{ij} - \ua_{ij} | \leq \delta
  \end{equation}
Finally, infinite precision is achieved outside the active set when the active set is correctly estimated. Indeed, let \(j\in \overline{I(\hat{t})} = \overline{\Ia} \). Then \(\hat{\nu}_j = \nua_j = 0\).
\end{proof}

\subsection{Computing dual variables from the consensus variable}\label{sec:computev}

It has been assumed above that, given
%an active set \(I \) and knowing
a value of \(t\), it is possible to compute the corresponding dual parameters \( \nu_j(t) \) that satisfy KKT conditions~\eqref{eq:kkt2},~\eqref{eq:kkt3} and~\eqref{eq:kkt5} with active set \( I(t) \). In other words, the following optimization problem should be solved:
\begin{equation}\label{eq:compute_nu}
  \underset{\nu_j(t)>0}{\text{Find}} \sum_i \left[|x_{ij}| - \lambda\nu_j(t)\right]^+ = t
\end{equation}
which has a single solution because the left-hand side is strictly decreasing. Algorithm~\ref{alg:routine} details this computation, while the rest of this subsection provides explanations.

First suppose \(x:=X_j\) is a single column of \( X \) belonging to the active set \( I(t) \) which has been sorted reversely using absolute values (columns which are not in the active set \( I(t) \) are not useful in order to compute non-zero \( \nu_j(t) \)). According to the KKT conditions, the proximal operator will threshold \( x \). Suppose also that the \( \ell_1 \) norm \( t \) of \( x \) after thresholding is fixed (as is the case in step 6 of Algorithm~\ref{alg:Algorithm1}). Furthermore, suppose again that the set \( \Omega_{\lambda\nu(t)} \) of non-clipped entries in \( x \) is known. With all these fixed quantities, it is simple to compute \( \nu(t) \) using~\eqref{eq:link}. In the following, I denote \( \Omega(i)=[1,i] \) a given, arbitrary set of non-clipped entries and \( \nu(t,i) \) the resulting dual parameter as computed with~\eqref{eq:link}, while \( \nu(t) \) is defined as in~\eqref{eq:defnu} and \( \Omega_{\lambda\nu(t)} \) follows the definition in~\eqref{eq:omegadef}. All the \(j\) indices are removed for simplicity in this subsection, because only one column \( x=X_j \) with \( j \) in \( I(t) \) is considered.

Now within Algorithm~\ref{alg:Algorithm1}, assuming \( t \) is fixed at each iteration is natural because of the bisection strategy. However the value of \( \Omega_{\lambda\nu(t)}\) is not known. What I propose is to compute all the possible values \( \nu(t,i) \) for each candidate \( \Omega(i) \). This is easily done because the computation of \( \nu(t,i) \) requires only the partial cumulative sum of \( x \) up to index \( i \), and therefore all the partial cumulative sums may only be computed once for each column of \(X\) at the beginning of Algorithm~\ref{alg:Algorithm1}. This costs exactly \(nm\) sums, outside the main loop of Algorithm~\ref{alg:Algorithm1}.

Then, any wrong choice of \( \Omega(i) \) must lead to a contradiction when performing the test
  \begin{equation}\label{eq:test}
    |x_{i+1}| \leq \lambda\nu(t,i) \;\text{ and }\; |x_i| > \lambda\nu(t,i)
  \end{equation}
  except for \(i=n\) for which only the first inequality should be checked. This holds since \( \nu(t,i) \) can only be the solution to~\eqref{eq:compute_nu} for a single, unknown \( i \), \textit{i.e.} the fact that \( \Omega(i) \) is indeed equal to \( \Omega_{\lambda\nu(t,i)} \) as defined in~\ref{eq:omegadef} is the only assumption that can be violated by contradiction.

Obviously both conditions cannot be violated at the same time for a (reverse) sorted \( x \): either \( \nu(t,i) \) is smaller than the true solution \( \nu(t) \) and \( |x_{i+1}| > \lambda\nu(t,i)  \), or it is too large and \( |x_{i}| \leq \lambda\nu(t,i) \). In fact, denoting \( i^* \) the index such that \( \nu(t,i^\ast)=\nu(t) \), if the set \( \Omega(i) \) is larger than the set of non-clipped entries \(\Omega(i^*) \) in \( S_{\lambda\nu(t)}(x) \), \textit{i.e.} when \( i>i^* \), then \( \nu(t,i) > \nu(t) \).
A proof is given at the end of this section.

Therefore, to find \( i^\ast \), it is sufficient to look for the largest index \( i \) such that
\begin{equation}
   |x_i| - \lambda\nu(t,i) \geq 0 .
\end{equation}
It is very important to notice at this stage that the sequence \( \left[x_{i} - \lambda\nu(t,i)\right]_i \) is provably sorted (see proof below). Therefore, finding \( i^\ast \) boils down to searching for the position of zero inside the sorted array \( x - \lambda\nu(t,:) \), which has complexity at worse \( \mathcal{O}(\log_2(n))\)\footnote{An earlier version of this work did not utilize this fact and searched an index satisfying~\eqref{eq:test}. It is therefore considerably slower and should not be used.}.

\begin{proof}
  Let me prove the two statements used to derive Algorithm~\ref{alg:routine}:
  \begin{itemize}
     \item[(i)] \( i>i^* \) implies \( \nu(t,i) > \nu(t) \),
     \item[(ii)] the sequence \( \left[x_{i} - \lambda\nu(t,i)\right]_i \) is reverse sorted.
   \end{itemize}

   (i) By applying~\eqref{eq:link} twice for \( \Omega(i) \) and \( \Omega(i^\ast) \) and taking the difference, it holds that
   \begin{equation}
     0 = \sum_{p=i^\ast +1}^{i} |x_p| + \lambda\left( i\nu(t,i) - i^\ast\nu(t)  \right).
   \end{equation}
   By adding and subtracting \( (i-i^\ast) \nu(t)\), the above is equivalent to
   \begin{equation}
     0 = \lambda i (\nu(t,i) - \nu(t)) + \sum_{p=i^\ast+1}^{i}\left[|x_p| - \lambda \nu(t)\right].
   \end{equation}
   By definition of \( i^\ast \), for any \( p>i^\ast \), \( |x_p| < \lambda\nu(t) \) and therefore \( \nu(t,i)>\nu(t) \).

   (ii) Denote \( N_i = |x_i| + \frac{t - \sum_{p=1}^{i}|x_p|}{i\lambda} = |x_i| + \frac{\sum_{p=i+1}^{n}|x_p|}{i\lambda} \). Let us check that \( N_i \geq N_{i+1} \). It holds that
   \begin{align*}
     N_i - N_{i+1} = |x_i| - |x_{i+1}| + \frac{1}{i(i+1)\lambda}\left[ \sum_{p=i+2}^{n} |x_p| + (i+1) |x_{i+1}| \right]
   \end{align*}
   which is nonnegative since \( |x_i| \geq |x_{i+1}| \).
\end{proof}
\paragraph{Complexity}
Ignoring line 1, the whole complexity of Algorithm~\ref{alg:routine}, applied to all column in \( I(t) \), is a vectorized matrix-matrix elementwise difference and division, plus a columnwise sorted search. In principle the cost should be dominated by a \( \mathcal{O}(nm) \) term from the elementwise division, however efficient low-level implementations of vectorized computations make this statement quite unrealistic. Moreover, all columns in the active set \( I(t) \) for a given value of \( t \) can be processed in parallel with matrix-level operations to avoid loops. Good implementations on modern computers with efficient matrix-level routines could consider this algorithm with a smaller \( \mathcal{O}(m\log_2(n)) \) with the sorted search being the bottleneck.

%\begin{itemize}
%  \item With known \(t\) and fixed \( \Omega(k)=[1,k]\) (recall that the sets \(\Omega_{\lambda\nu_j} \) are telescopic), using expression~\ref{eq:link2}, a value \(\nu_j(t,k_j)\) is obtained. Note that \(t, \{\nu_j(t,k_j)\}_{j\in I(t)} \) need not be in \(\Sigma(I(t))\).
%  \item The optimal \(k_j(t)\) integers such that %\(\Omega_{\lambda\nu_j(t)} = [1,k_j(t)]\)
%  \(\nu_j(t) = \nu_j(t,k_j(t)) \)
%  are not known. But to compute the values of \(\nu_j(t,k_j)\) for all possible \(k_j\), it is only required to compute all the partial cumulative sums \( \sum_{i=1}^{p}{|x_{ij}|} \) with \(p\in[1,n]\). This costs exactly \(nl\) sums and can be done once at the beginning of Algorithm~\ref{alg:Algorithm1}.
%  \item For all \(j\), there can only exist one optimal \(k_j(t)\) (given that \(\nu_j(t)> 0\), otherwise this discussion is pointless). A sufficient test for finding whether a candidate \(k_j\) is indeed the correct one is
%  \begin{equation}\label{eq:test}
%    |x_{k_j j}| > \lambda\nu_j([1,k_j]) \text{ and } |x_{(k_j +1) j}| \leq \lambda\nu_j([1,k_j])
%  \end{equation}
%  except for \(k_j=n\) for which only the first inequality should be checked. This step costs at worse \(2nl\) comparisons.
%\end{itemize}

\begin{algorithm}
  \caption{Algorithm for finding \(\nu(t)\), for a single column}\label{alg:routine}
  \begin{algorithmic}
   \STATE{\textbf{Inputs}: t, \(\lambda \), column-wise sorted column \(x\in\mathds{R}^{n\times 1}\)}
   \STATE{1. Set \(z\) as the vector of partial cumulative sums of \(x\) (this can be precomputed).}
   \STATE{2. Define \(N := x - \frac{z-t}{\lambda[1;\ldots;n]}\), division understood element-wise.}
   \STATE{3. Find \(i\) such that \(N[i]\geq 0 > N[i+1]\) using a sorted list search.
   \STATE{4.} Set \(\nu(t) := \nu(t,i) \).}
   \STATE{\textbf{Outputs}: \(\nu(t)\).}
  \end{algorithmic}
\end{algorithm}

\subsection{Choice of regularization parameter}

In a context of using \(\prox_{\lambda\ell_{1,1}} \) inside an optimization algorithm such as a proximal gradient, it would be useful to be provided with some guidance on how to choose \(\lambda \). We show in what follows that there exist a maximal value of \(\lambda \) above which the solution is null.
\begin{proposition}\label{prop:lambdamax}
  Set \(\lambda_{\max}:=\sum_{j\leq m} \max_{i\leq n} |x_{ij}| \), and suppose \( X \) has no zero columns. Then
  \(\lambda \geq \lambda_{\max} \) if and only if \(\prox_{\lambda\ell_{1,1}}(X)=0\). Furthermore, if any column of \(U^*\) is null, then \(\lambda\geq\lambda_{\max}\) and \(U^*\) is null.
\end{proposition}
\begin{proof}
Suppose \(\prox_{\lambda\ell_{1,1}}(X)=0\). Then all constraints are active (unless \(X\) has a zero column) and columns are fully clipped. Hence for all \( j \leq m \),
\begin{equation}
  \nua_j\lambda \geq \max_i |x_{ij}|
\end{equation}
which, using the fact that \(\sum_{j\leq m} \nua_j = 1\), yields
\begin{equation}
  \lambda \geq \sum_{j\leq m} \max_{i\leq n} |x_{ij}|~.
\end{equation}
Conversely, if \(\prox_{\lambda\ell_{1,1}}(X)\neq 0\), then for all \(j\) such that \(X_j\neq 0\) there exist \(i_j\) such that\footnote{the reasoning here is that if a nonzero column of \(X\) is clipped to zero in the proximal operator, then it must be in the active set and \(t=0\) as well, which is impossible by assumption. Therefore no nonzero column of \(X\) is clipped to zero.}
\begin{equation}
  \nua_j\lambda < |x_{i_j j}| < \max_i |x_{ij}|
\end{equation}
which yields
\begin{equation}
  \lambda < \sum_{j} \max_i |x_{ij}|~.
\end{equation}

To prove the second part of the proposition, notice that if any column of \(U^*\) is null, then it is clipped, therefore \(t^*=0\), thus \(U^*=0\).
\end{proof}

\section{Discussion}

\subsection{Applications}
It has been shown that the proximal operator of two induced matrix norms \(\ell_{1,1}\) and \(\ell_{\infty,\infty}\) can be computed for a reasonable cost at arbitrary precision. However, it remains to motivate the study of this operator. Why should one care about the proximity operator of induced matrix norms?

A typical use of proximal operator, at least which the author is familiar with, is as a means to solve regularized optimization problems of the form
\begin{equation}\label{eq:genoptpb}
  \underset{X}{\argmin}{f(X)+g(X)}
\end{equation}
where \(f\) is a smooth convex function, and \(g\) is convex proper lower-semicontinuous and admits a computable proximal operator. Then the solution of~\eqref{eq:genoptpb}, if it exists, may be found by a variant of choice of the proximal gradient algorithm, see for instance~\cite{Beck2017first} for more details. But such a problem with \(g=\ell_{1,1} \), as far as I know, does not appear in the literature. So for now, this works remain purely theoretical and aimed at extending the already long list of norms for which the proximal operator is easily computed.

\subsection{Behavior with respect to the input}

An important detail about \(\ell_{1,1}\) is that it is a ``uniform'' column-wise \(\ell_1\) norm on the columns of the input \(X\). In other words, using the \(\ell_{1,1}\) matrix induced norm as a regularizer in~\eqref{eq:genoptpb} penalizes more intensely the columns of the solutions that have larger \(\ell_1\) norms. This behavior is quite clear when studying the proximal operator: it may happen that one column of \(X\), say the first one, is much larger than any other. Then it may be the only one in the active set. Moreover, it can even hold that \( \prox_{\lambda\ell_{1,1}}{(X)}_1\) is \(1\)-sparse while \(\prox_{\lambda\ell_{1,1}}{(X)}_{j} = X_{j}\) for \(j\neq 1\) for a well-chosen \(\lambda \). For instance, setting
\begin{equation}
  X = \left[
  \begin{array}{cc}
  1 & 0.1 \\
  2 & 0.2 \\
  3 & 0.3 \\
  \end{array}
  \right],
\end{equation}
one easily checks, for instance using the KKT conditions, that
\begin{equation}
  \prox_{2.1\ell_{1,1}}(X) = \left[
  \begin{array}{cc}
  0 & 0.1 \\
  0 & 0.2 \\
  0.9 & 0.3
  \end{array}
  \right]~.
\end{equation}
However, columns of the proximal operator cannot be completely thresholded to zero unless \(\lambda \) is larger than \(\lambda_{\max}\) defined in Proposition~\ref{prop:lambdamax}. This means that, unlike the \(\ell_1\) norm computed on all the elements of \(X\), the \(\ell_{1,1}\) norm will not allow a strict subset of columns to be forced to zero.
On the other hand, if the columns of \(X\) have balanced \(\ell_1\) norms, \textit{a priori} the proximal operator will put to zero entries in all columns of \(X\). These are quite remarkable features of this regularization which differ from, for instance, the \(\ell_1\) norm of the unfolded \(X\), or the sum of the \(\ell_1\) norms of columns of \(X\).

\subsection{Practical runtime}

To showcase the computation time of the proposed Algorithm~\ref{alg:Algorithm1}, a few experiments are ran below. The proximal operator is implemented in Python without any particular fine tuning. The code is available at \url{github.com/cohenjer/Tensor_codes}, a Matlab implementation is also provided.

\begin{figure}
  \centering
  \includegraphics[width=10cm]{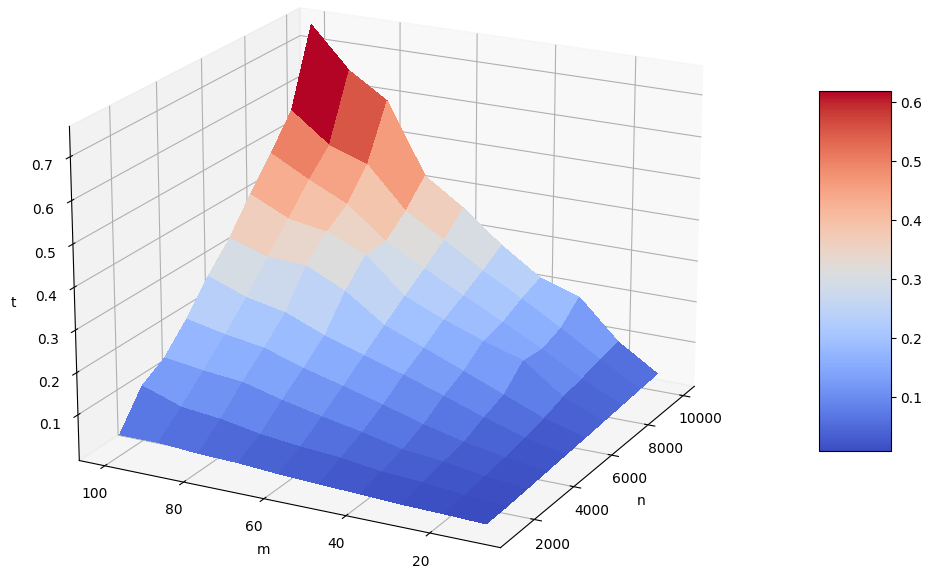}
  \caption{Average proximal operator \(\prox_{\ell_{1,1}}\) computation time. Parameter \(\lambda \) is set to half the maximum regularization, which yields a high sparsity level in the solution.}\label{fig:tpscalcul}
\end{figure}

For various sizes \((n,m)\), the average computation time of the \(\ell_{1,1}\) proximal operator of \(X\in\mathds{R}^{n\times m}\), sampled element-wise from a unitary centered Gaussian distribution, is recorded. Figure~\ref{fig:tpscalcul} shows the raw average computation time \(t(n,m)\) in seconds where the average is taken over 5 realizations for \(X\in\mathds{R}^{n\times m}\). Parameters \(\lambda \) and \(\delta \)
where set respectively to \(0.5\lambda_{\max}\) and \(10^{-8}\). It can be observed that the computation time seems fairly linear with respect to \(m\) and \(n\) as expected.

\section*{Acknowledgments}
I want to thank Le Thi Khanh Hien for a very helpful proof-checking and discussion around the uniqueness of the dual parameters, as well as for pointing towards the finite minimax literature.

\bibliographystyle{plain}

\end{document}